\documentclass[12pt]{amsart}

\input xy
\usepackage[english]{babel}
\usepackage[latin1]{inputenc}
\usepackage{graphicx}
\usepackage{amsmath}
\usepackage{amssymb}
\usepackage{amscd}
\usepackage{color}
\usepackage{oldgerm}
\usepackage{amsfonts}
\usepackage{newlfont}
\usepackage{longtable}
\usepackage{multirow}
\usepackage{xcolor}

\DeclareOldFontCommand{\rm}{\normalfont\rmfamily}{\mathrm}

\usepackage[all]{xy}

\usepackage{upref}

\def\F{\Bbb F}
\def\Z{\Bbb Z}

\def\ad{\operatorname{ad}}

\def\Der{\operatorname{Der}}

\def\dim{\operatorname{dim}}

\def\End{\operatorname{End}}

\def\Hom{\operatorname{Hom}}

\def\g{\frak g}
\def\gl{\frak{gl}}
\def\h{\frak h}

\def\a{\frak{a}}

\theoremstyle{plain}\swapnumbers

\newtheorem{Theorem}{Theorem}[section]
\newtheorem{lema}[Theorem]{Lemma}
\newtheorem{proposicion}[Theorem]{Proposition}

\newtheorem{cor}[Theorem]{Corollary}

\newtheorem{rmk}[Theorem]{Remark}

\newtheorem{claim}[Theorem]{Claim}

\begin{document}

\title[Homogeneous quadratic Lie super algebras]{Homogeneous quadratic Lie super algebras}

\author{R. Garc\'{\i}a-Delgado}

\address{Facultad de Ciencias, UASLP, Av. Salvador Nava s/n, Zona Universitaria, CP 78290; San Luis Potos\'{i}, M\'exico}
\email{rosendo.garciadelgado@alumnos.uaslp.edu.mx}


\maketitle

%%%%%%%%%%%%%%%%%%%%%%%%%%%%%%%%%
%%%%%%%%%%%%%%%

\begin{abstract}
In this work we state a version of the double extension for homogeneous quadratic Lie super algebras that includes even and odd cases. We prove that any indecomposable, non-simple and homogeneous quadratic Lie super algebra is obtained by means of this type of double extension. We also show that with this construction we can recover previously studied cases as well as some other which can not be recover with former versions of double extensions. 
\end{abstract}

%%%%%%%%%%%%%%%%%%%
%%%%%%%%%%%%%%%%%%%%%     SECTION  INTRODUCTION 
%%%%%%%%%%%%%%%%%%%

\section*{Introduction}
A {\it homogeneous quadratic Lie super algebra} is a triple $(\g,[\,\cdot\,,\,\cdot\,],B)$, where $(\g,[\,\cdot\,,\,\cdot\,])$ is a 
Lie super algebra and $B:\g \times \g \to \F$ is an invariant, non-degenerate, super-symmetric, homogeneous 
and bilinear form, called it {\it invariant metric}. When the context is clear, we refer to a quadratic Lie super algebra only by $\g$.
\smallskip

For classical quadratic Lie algebras there are well-known procedures to construct such structures, One of these is the double extension (see \cite{Med}). In addition, Medina-Revoy theorem proves that any indecomposable non-simple quadratic Lie algebra, is a double extension of another quadratic Lie algebra of smaller dimension. Another approach can be found in \cite{Kath}, where the authors define two canonical ideals for a quadratic Lie algebra without simple ideals. Then the authors construct certain cohomology set that allows them to get classification results. 
\smallskip

Concerning to Lie super algebras, in \cite{Ben} the authors state a notion of double extension for even quadratic Lie super algebras, and prove that any indecomposable, non-simple and even quadratic Lie super algebra can be constructed using this procedure, which they called it \emph{generalized double extension}. Just as in the classical case, the generalized double extension gives a way to construct an even quadratic Lie super algebra on a super-space of the form $\g=\a \oplus \h \oplus \a^{*}$, where $\h$ is an even 
quadratic Lie super algebra and $\a$ is a Lie super algebra (see \cite{Ben}, {\bf Thm. 2.3}).
\smallskip

With regard to odd quadratic Lie super algebras, in \cite{Alb} is carried out similar ideas as in the even case. Let $\mathfrak{h}$ be an odd quadratic Lie super algebra and $\a$ be a one-dimensional super-space with odd-generator. In \cite{Alb} appears the conditions to set an odd quadratic Lie 
super algebra structure on the super-space $\mathfrak{a} \oplus \mathfrak{h} \oplus \mathbf{P}(\mathfrak{a}^{*})$, 
where $\mathbf{P}$ is the change of parity operator (see {\bf Thm 4.2} in \cite{Alb} and \textbf{Prop. \ref{propiedades de P}} below).
\smallskip

There are quadratic Lie super algebras that can not be obtained 
from the results of \cite{Alb} and \cite{Ben}. For instance, the Heisenberg Lie super algebra 
with odd bilinear form extended by an even derivation (see \cite{Mac}). The underlying
super-space of this Lie super algebra is of the form: $\F D \oplus \mathfrak{h} \oplus \F \hslash$, 
where $|D|=0$, $|\hslash|=1$, $\mathfrak{h}$ is a non-degenerate sub-super-space and the structure 
on $\mathfrak{h} \oplus \F \hslash$ is of a central extension of $\mathfrak{h}$ by $\F \hslash$. 
\smallskip

The aim of this work is to provide a version of the generalized double extension that includes the odd case. We proceed as in \cite{Ben}, with the difference that we state definitions and results that cover the odd case. We state the notion of \emph{$\delta$-coadjoint representation} of a Lie super algebra (see \textbf{Prop. \ref{prop delta coadjunta}}), which is a version of the coadjoint representation that deals with even and odd cases. In \textbf{Lemma \ref{lema contexto}} we give the notion of \emph{$\delta$-context of generalized double extension} -- similar to the one given in \cite{Ben}--, and in \textbf{Thm. \ref{teorema doble extension super}}, we state a version of the generalized double extension that covers the odd case.
\smallskip

We prove that any non-simple, indecomposable and homogeneous quadratic Lie super algebra can be identified with this version of double extension (see \textbf{Cor. \ref{corolario 1}}). Finally, we prove that with the results obtained here, we can recover the odd case studied in \cite{Alb}, as well as the Heisenberg Lie super algebra with odd bilinear form extended by an even derivation (see \cite{Mac}).

\section{Basic definitions}

A vector space $V$ is \textbf{$\mathbb{Z}_2$-graded}, if for each $\delta$ in $\Z_2=\{0,1\}$, there is a subspace $V_{\delta}$ such that $V=V_0 \oplus V_1$. A non-zero element $u$ in $V$ is 
\textbf{homogeneous of degree $|u|$}, if $u$ belongs to $V_{|u|}$, where $|u|$ is in $\Z_2$. A $\mathbb{Z}_2$-graded space is called \textbf{super-space}. We assume the trivial graduation on the ground field $\mathbb{F}$, that is $\mathbb{F}_0=\mathbb{F}$ and $\mathbb{F}_1=\{0\}$.
\smallskip

The \textbf{dimension} of a super-space $V=V_0 \oplus V_1$ is its dimension as a vector space, that is $\dim(V)=\dim(V_0)+\dim(V_1)$.
\smallskip

From now on we only consider homogeneous elements in definitions and statements.
\smallskip

A \textbf{Lie super algebra} is a super-space $\mathfrak{g}=\g_0 \oplus \g_1$, with a bilinear map $[\,\cdot\,,\,\cdot\,]:\mathfrak{g} \times \mathfrak{g} \rightarrow \mathfrak{g}$, satisfying:
\smallskip

\textbf{(i)} $[\mathfrak{g}_{\eta},\mathfrak{g}_{\zeta}] \subseteq \mathfrak{g}_{\eta+\zeta}$, for all $\eta,\zeta$ in $\mathbb{Z}_2$.
\smallskip

\textbf{(ii)} The bilinear map $[\,\cdot\,,\,\cdot\,]$ is \textbf{super skew-symmetric}: 
$$
[x,y]=-(-1)^{|x||y|}[y,x], \quad \text{ for all }x,y \text{ in }\mathfrak{g}.
$$

\textbf{(iii)} The bilinear map $[\,\cdot\,,\,\cdot\,]$ satisfies the \textbf{Jacobi super identity}:
\begin{equation}\label{Jacobi super identity}
\begin{split}
& \sum_{\text{cyclic}}(-1)^{|x||z|}[x,[y,z]]\\
&=(-1)^{|x||z|}[x,[y,z]]+(-1)^{|y||x|}[y,[z,x]]+(-1)^{|z||y|}[z,[x,y]]=0,
\end{split}
\end{equation}
A blinear map $[\,\cdot\,,\,\cdot\,]$ satisfying \textbf{(i),(ii),(iii)} is called \textbf{bracket}. 
\smallskip

An \textbf{ideal} $I$ of a Lie super algebra $\g$ is a graded subspace $I=I_0 \oplus I_1$, of $\g$ such that $[\g,I] \subset I$.
\smallskip

Let $\mathfrak{g}$ be a Lie super algebra and let $B:\mathfrak{g} \times \mathfrak{g} \rightarrow \mathbb{F}$ be a bilinear form.
\smallskip

\textbf{(i)} $B$ is \textbf{super-symmetric} if $B(x,y)\!\!=\!\!(-1)^{|x||y|}B(y,x)$ for all $x,\!y$ in $\mathfrak{g}$.
\smallskip

\textbf{(ii)} $B$ is \textbf{invariant} if $B([x,y],z)=B(x,[y,z])$, for all $x,y,z$ in $\mathfrak{g}$.
\smallskip

\textbf{(iii)} $B$ is \textbf{even} if $B(\mathfrak{g}_0,\mathfrak{g}_1)=B(\mathfrak{g}_1,\mathfrak{g}_0)=\{0\}$. In this case we write $|B|=0$.
\smallskip

\textbf{(iv)} $B$ is \textbf{odd} if $B(\mathfrak{g}_{\eta},\mathfrak{g}_{\eta})=\{0\}$, for all $\eta$ in $\mathbb{Z}_2$. In this case we write $|B|=1$.
\smallskip

\textbf{(v)} The bilinear form $B$ is \textbf{homogeneous} if either it is even or odd.
\smallskip

If $B$ is non-degenerate, invariant, super-symmetric and homogeneous then $B$ is said to be an \textbf{invariant metric} on $\mathfrak{g}$.
\smallskip

A {\bf quadratic Lie super algebra of $\delta$-degree}, is a triple $(\g,[\,\cdot\,,\,\cdot\,],B)$, where $(\g,[\,\cdot\,,\,\cdot\,])$ is a Lie super algebra and $B$ is an invariant metric of $\delta$-degree, that is $|B|=\delta$. %If $\delta=0$ (resp. $\delta=1$), we say that $\g$ is an \emph{even} quadratic Lie super algebra (resp. \emph{odd} quadratic Lie super algebra).
\smallskip

Let $B$ be a non-degenerate, super-symmetric and homogeneous bilinear form on a super-space $\g$.
Let $B^{\flat}:\g\to\g^\ast$ be the map defined by
$B^{\flat}(x)\,(y)=B(x,y)$, for all $x,y$ in $\g$, then 
$|B^{\flat}(x)|=|B|+|x|$. A homogeneous linear map $D:\g \to \g$, is \textbf{skew-symmetric with respect to $B$} if: $B(D(x),y)=-(-1)^{|x||D|}B(x,D(y))$. The super-space generated by skew-symmetric maps with respect to $B$, is denoted by $\mathfrak{o}(B)$.
\smallskip

If a Lie super algebra $\g$ admits an invariant metric $B$, then
$B^{\flat}$ is an isomorphism of \emph{$\g$-modules}, that is: $B^{\flat} \circ \ad(x)= (-1)^{|x||B|}\ad^\ast\!(x) \circ B^{\flat}$, where $\ad^{\ast}:\g \to \gl(\g^{\ast})$ is the \emph{coadjoint representation} of $\g$, defined by: 
\begin{equation}\label{representacion coadjunta}
\ad^\ast(x)(\theta)=-(-1)^{|x||\theta|}\,\theta\circ\ad(x),\,\text{ for all }x \in \g,\text{ and }\theta \in \g^{\ast}.
\end{equation}

A \textbf{derivation} of a Lie super algebra $\mathfrak{g}$, is a homogeneous linear map $D:\g \to \g$, such that: $D([x,y])=[D(x),y]+(-1)^{|D||x|}[x,D(y)]$, for all $x,y$ in $\g$. The vector super-space generated by all derivations of $\g$, is denoted by $\Der(\mathfrak{g})$.

\subsection{$\delta$-coadjoint representation}

Let $V=V_0 \oplus V_1$ be a super-space. Let $\mathbf{P}(V)$ be the super-space $\mathbf{P}(V)=\mathbf{P}(V)_0 \oplus \mathbf{P}(V)_1$, where $\mathbf{P}(V)_0=V_1$, and $\mathbf{P}(V)_1=V_0$.
\smallskip

The map $V \mapsto \mathbf{P}(V)$ is called \textbf{change of parity operator.} As vector spaces, $V$ and $\mathbf{P}(V)$ are the same but as super-spaces they are different. In order to distinguish the elements of $V$ from that of $\mathbf{P}(V)$, we write $P(v)$ to denote an element of $\mathbf{P}(V)$, for a given $v$ in $V$. If $v$ is a homogeneous element of $V$, the graded of $P(v)$ is $|P(v)|=|v|+1$.
\smallskip

Let $\delta$ be in $\Z_2$. We define the super-space $\mathbf{P}_{\delta}(V)$, by:
\begin{equation}\label{48}
\mathbf{P}_{\delta}(V)=\left\{
\begin{array}{cl}
V&\mbox{ if }\delta=0\\
\mathbf{P}(V)&\mbox{ if }\delta=1.
\end{array}\right.
\end{equation}

We have the following properties of $\mathbf{P}_{\delta}$.

\begin{proposicion}\label{propiedades de P}{\sl
Let $U,V$ and $W$ be super-spaces.
\smallskip

\textbf{(i)} $\mathbf{P}_{\delta}\left(\mathbf{P}_{\delta}(V) \right)=V$.
\smallskip

\textbf{(ii)} $\mathbf{P}_{\delta}(\Hom\left(V,W)\right)=\Hom\left(\mathbf{P}_{\delta}(V),W\right)=\Hom(V,\mathbf{P}_{\delta}(W))$.
\smallskip

\textbf{(iii)} Let $\mathcal{L}^2(U;V;W)$ be the super-space generated by homogeneous bilinear maps $f\!:\!U \!\times \!V \!\to \!W$. Then $\mathbf{P}_{\delta}(\mathcal{L}^2(U;V;W))\!=\!\mathcal{L}^2(U;V;\mathbf{P}_{\delta}(W))$.
\smallskip

\textbf{(iv)} $\mathbf{P}_{\delta}(V^{*})=\left(\mathbf{P}_{\delta}(V)\right)^{*}$.}

\end{proposicion}
\begin{proof}
For $\delta\!=\!0$ the results are clear, we assume $\delta\!=\!1$. Let $\eta$ be in $\Z_2$.
\smallskip

\textbf{(i)} From $\mathbf{P}(\mathbf{P}(V)_{\eta})=\mathbf{P}(V_{\eta+1})=V_{\eta+2}=V_{\eta}$, it follows $\mathbf{P}(\mathbf{P}(V))=V$.
\smallskip

\textbf{(ii)} Let $T$ be in $\Hom(V,W)_0$, then $P(T)$ is in $\mathbf{P}(\Hom(V,W))_0=\Hom(V,W)_1$, that is $P(T)(V_{\eta})\subset W_{\eta+1}=\mathbf{P}(W_{\eta})$, hence $P(T)$ belongs to $\Hom(V,\mathbf{P}(W))_0$, which proves that $\mathbf{P}(\Hom(V,W))_0 \subset \Hom(V,\mathbf{P}(W))_0$. Similarly, from $P(T)(V_{\eta})=P(T)(\mathbf{P}(V)_{\eta+1})\subset W_{\eta+1}$, we get that $P(T)$ belongs to $\Hom(\mathbf{P}(V),W)_0$, then $\mathbf{P}(\Hom(V,W))_0 \subset \Hom(\mathbf{P}(V),W))_0$.
\smallskip

Similar argument proves that: $\mathbf{P}(\Hom(V,W))_1 \subset \Hom(V,\mathbf{P}(W))_1$, and $\mathbf{P}(\Hom(V,W))_1 \subset \Hom(\mathbf{P}(V),W))_1$; thus, 
$$
\mathbf{P}(\Hom(V,W))\! \subset \!\Hom(V,\mathbf{P}(W)),\,\text{ and }\,\mathbf{P}(\Hom(V,W)) \!\subset \!\Hom(\mathbf{P}(V),W).
$$
Since $V$ and $W$ are finite dimensional, we deduce that $\mathbf{P}(\Hom(V,W))=\Hom(V,\mathbf{P}(W))=\Hom(\mathbf{P}(V),W)$.
\smallskip

\textbf{(iii)} Observe that $\mathcal{L}^2(U;V,W)=\Hom(U,\Hom(V,W))$. Applying \textbf{(ii)} we obtain:
$$
\aligned
\mathbf{P}(\mathcal{L}^2(U;V,W))&=\mathbf{P}(\Hom(U,\Hom(V,W)))=\Hom(U,\mathbf{P}(\Hom(V,W)))\\
\,&=\Hom(U,\Hom(V,\mathbf{P}(W)))=\mathcal{L}^2(U;V;\mathbf{P}(W)).
\endaligned
$$  

\textbf{(iv)} Due to $V^{\ast}=\Hom(V,\F)$, then by \textbf{(ii)} we get:
$$
\mathbf{P}(V^{\ast})=\mathbf{P}(\Hom(V,\F))=\Hom(\mathbf{P}(V),\F)=\mathbf{P}(V)^{\ast}.
$$
\end{proof}

\begin{rmk}\label{RMK}{\rm
If $T$ is an element of $\Hom(V,W)$, by \textbf{Prop. \ref{propiedades de P}.(ii)}, $P(T)$ is an element of $\Hom(\textbf{P}(V),W)$, defined by:
$$
P(T)(P(v))=T(v),\qquad \text{ for all }v \in V.
$$
If $T$ is homogeneous, then $P(T)$ and $T$ are different for they have different degrees.}
\end{rmk}

In the next result we introduce the coadjoint representation on $\mathbf{P}_{\delta}(\g)^{\ast}$.

\begin{proposicion}\label{prop delta coadjunta}{\sl
Let $(\g,[\,\cdot\,,\,\cdot\,])$ be a Lie super algebra over a field $\F$ and let $\ad^{\ast}:\g \to \gl(\g^{\ast})$ be the coadjoint representation of $\g$. For every $x$ in $\mathfrak{g}$ and $f$ in $\g^{*}$, let $\ad_{\boldsymbol{\delta}}^{\ast}(x)(P_{\delta}(f))$ be in $\mathbf{P}_{\delta}(\g)^{\ast}$ defined by:
\begin{equation}\label{coadjunta}
\ad_{\delta}^{\ast}(x)\left(P_{\delta}(f)\right)\left(P_{\delta}(y)\right)=-(-1)^{(|f|+\delta)|x|}f([x,y]), \quad \text{ for all }y \in \g,
\end{equation}
Then, \eqref{coadjunta} yields a representation $\ad_{\delta}^{\ast}:\g \to \gl\left(\mathbf{P}_{\delta}(\g)^{\ast}\right)$, of $\g$ on $\mathbf{P}_{\delta}(\g)^{\ast}$, such that:
\begin{equation}\label{relacion coadjuntas}
\ad_{\delta}^{\ast}(x)\circ P_{\delta}=(-1)^{\delta |x|}P_{\delta}\circ \ad^{\ast}(x),\quad \text{ for all }x \in \g.
\end{equation}
}
\end{proposicion}
\begin{proof}
First we will prove \eqref{relacion coadjuntas}. For $\delta=0$, $\ad^{\ast}$ and $\ad^{\ast}_{\boldsymbol{\delta}}$ are the same (see \eqref{representacion coadjunta}), so we assume $\delta=1$. Let $x,y$ be in $\g$ and $f$ be in $\g^{\ast}$. Using \eqref{representacion coadjunta} and \eqref{coadjunta} we get:
\begin{equation}\label{super1}
\ad_{\mathbf{1}}^{\ast}(x)(P(f))(P(y))\!=\!-(\!-1)^{(|f|+1)|x|}\!f([x,y])\!=\!(\!-1)^{|x|}\!\ad^{\ast}\!(x)(f)(y).
\end{equation}
We know that $\ad^{\ast}(x)(f)$ belongs to $\Hom(\g;\F)$. Then by \textbf{Prop. \ref{propiedades de P}.(ii)}, $P(\ad^{\ast}(x)(f))$ belongs to $\mathbf{P}(\Hom(\g;\F))=\Hom(\mathbf{P}(\g);\F)$. Hence, by \textbf{Remark \ref{RMK}} we get: $P(\ad^{\ast}(x)(f))(P(y))=\ad^{\ast}(x)(f)(y)$. We substitute this in \eqref{super1}, to obtain:
$$
\ad_{\mathbf{1}}^{\ast}(x)(P(f))(P(y))=(-1)^{|x|}P(\ad^{\ast}(x)(f))(P(y)),
$$
from where we deduce \eqref{relacion coadjuntas} for $\delta=1$. Now we shall prove that $\ad_{\mathbf{1}}^{\ast}$ is a representation of $\g$ on $\mathbf{P}(\g)^{\ast}$. From \eqref{relacion coadjuntas} and \eqref{super1} we get:
$$
\aligned
\ad_{\mathbf{1}}^{\ast}(x)\circ \ad_{\mathbf{1}}^{\ast}(y)(P(f))&=(-1)^{|y|}\ad_{\mathbf{1}}^{\ast}(x)(P(\ad^{\ast}(y)(f)))\\
\,&=(-1)^{|x|+|y|}P\left(\ad^{\ast}(x)\circ \ad^{\ast}(y)(f)\right).\\
\endaligned
$$
Similarly we have: 
$$
\ad_{\mathbf{1}}^{\ast}(y)\circ \ad_{\mathbf{1}}^{\ast}(x)(P(f))=(-1)^{|x|+|y|}P\left(\ad^{\ast}(y)\circ \ad^{\ast}(x)(f)\right)
$$
Then, by \eqref{relacion coadjuntas} it follows
\begin{equation}\label{super3}
\begin{split}
[\ad_{\mathbf{1}}^{\ast}(x),\ad_{\mathbf{1}}^{\ast}(y)]_{\gl(\mathbf{P}(\g^{\ast}))}(P(f))&=(-1)^{|x|+|y|}P\left(\ad^{\ast}([x,y])(f)\right)\\
\,&=\ad_{\mathbf{1}}^{\ast}([x,y])(P(f)).
\end{split}
\end{equation}
Which proves that $\ad^{\ast}_{\mathbf{1}}$ is a representation of $\g$ on $\mathbf{P}(\g)^{\ast}$.
\end{proof}

We call $\ad^{\ast}_{\boldsymbol{\delta}}$, the \textbf{$\delta$-coadjoint representation} of $(\g,[\,\cdot\,,\,\cdot\,])$.

\subsection{Generalized semi-direct product}

Let $(\a,[\,\cdot\,,\,\cdot\,]_{\a})$ and $(\h,[\,\cdot\,,\,\cdot\,]_{\h})$ be Lie super algebras. Let $\Theta:\a \to \Der(\h)$ be an even linear map and $\Lambda:\a \times \a \to \h$ be an even super antisymmetric bilinear map, such that:
\begin{align}
\label{semidirect 1} & [\Theta(x),\Theta(y)]_{\gl(\h)}-\Theta([x,y]_{\a})=\ad_{\h}(\Lambda(x,y)),\quad \text{ and }\\
\label{semidirect 2} & \sum_{\text{cyclic}}(-1)^{|x||z|}\left(\Theta(x)(\Lambda(y,z))+\Lambda(x,[y,z]_{\a})\right)=0,
\end{align}
for all $x,y,z$ in $\a$. We define the bracket $[\,\cdot\,,\,\cdot\,]$ on the super-space $\g=\a \oplus \h$, by:
\begin{equation}\label{semidirect 3}
[x+u,y+v]\!=\![x,y]_{\a}\!+\!\Theta(x)(v)\!-\!(-1)^{|x||y|}\Theta(y)(u)\!+\!\Lambda(x,y)\!+\![u,v]_{\h},
\end{equation}
for all homogeneous elements $x+u,y+v$ of $\g$, that is $|x|=|u|$ and $|y|=|v|$. The bracket $[\,\cdot\,,\,\cdot\,]$ induces a Lie super algebra structure on $\g$, called it the \emph{generalized semi-direct product of $(\h,[\,\cdot\,,\,\cdot\,]_{\h})$ by $(\a,[\,\cdot\,,\,\cdot\,]_{\a})$, by means of $(\Theta,\Lambda)$ (see \cite{Dimitri} and \cite{Ben}).}

%%%%%%%%%%%%%%%%%%%%%%%%%%%%%%%%%%%%%%%%%%%%%%%%%%%%%%%%
\section{Generalized double extension of homogeneous quadratic Lie super algebras}
%%%%%%%%%%%%%%%%%%%%%%%%%%%%%%%%%%%%%%%%%%%%%%%%%%%%%%%%%%%

In this section we shall construct a homogeneous quadratic Lie super algebra on a super-space of the form $\a \oplus \h \oplus \textbf{P}_{\delta}(\a)^{\ast}$, where $(\a,[\,\cdot\,,\,\cdot\,]_{\h})$ is a Lie super algebra and $(\h,[\,\cdot\,,\,\cdot\,]_{\h},B_{\h})$ is a quadratic Lie super algebra of degree $\delta$, that is $|B_{\h}|=\delta$. The invariant metric on $\a \oplus \h \oplus \textbf{P}_{\delta}(\a)^{\ast}$, also will have degree $\delta$. 
\smallskip

First we need an even linear map $\rho:\a \to \Der(\h)\cap \mathfrak{o}(B_{\h})$ and an even super skew-symmetric bilinear map $\lambda:\a \times \a \to \h$, satisfying:
\begin{align}
\label{deh1}& [\rho(x),\rho(y)]_{\gl(\h)}-\rho([x,y]_{\a})=\ad_{\h}(\lambda(x,y)),\,\,\,\text{ and }\\
\label{deh2}& \sum_{\text{cyclic}}(-1)^{|z||x|}\left(\rho(x)(\lambda(y,z))\!+\!\lambda(x,[y,z]_{\a})\right)\!=\!0,\,\text{ for all }x,y,z \in \a.
\end{align}
Let $\chi_{\delta}:\a \times \h \to \mathbf{P}_{\delta}(\h)^{\ast}$ be the even bilinear map defined by:
\begin{equation}\label{chi delta}
\chi_{\delta}(x,u)(P_{\delta}(y))=-(-1)^{|u||y|}B_{\h}(\lambda(x,y),u),
\end{equation}
for all $x,y$ in $\a$ and $u$ in $\h$. If $\chi:\a \times \h \to \a^{\ast}$ is the bilinear map: $\chi(x,u)(y)=-(-1)^{|u||y|}B_{\h}(\lambda(x,y),u)$, then $|\chi|=\delta$ and $\chi_{\delta}=P_{\delta}(\chi)$.
\smallskip

Let $\omega_{\delta}:\a \times \a \to \mathbf{P}_{\delta}(\a)^{\ast}$ be an even super skew-symmetric bilinear map, satisfying:
\begin{equation}\label{deh3}
\sum_{\text{cyclic}}(-1)^{|z||x|}\!\left(\ad_{\delta}^{\ast}(x)(\omega_{\delta}(y,z))\!+\!\omega_{\delta}(x,[y,z]_{\a})\!+\!\chi_{\delta}(x,\lambda(y,z))\right)\!=\!0,
\end{equation}
and the super \emph{cyclic condition}:
\begin{equation}\label{super ciclica}
\omega_{\delta}(x,y)(P_{\delta}(z))=(-1)^{(|y|+|z|)|x|}\omega_{\delta}(y,z)(P_{\delta}(x)),
\end{equation}
for all $x,y,z$ in $\a$. Finally, we need an even super skew-symmetric bilinear map $\Phi_{\delta}:\h \times \h \to \mathbf{P}_{\delta}(\a)^{\ast}$, defined by:
\begin{equation}\label{deh4}
\Phi_{\delta}(u,v)(P_{\delta}(x))=(-1)^{|x|(|u|+|v|)}B_{\h}(\rho(x)(u),v),
\end{equation}
for all $x$ in $\a$ and $u,v$ in $\h$. If $\Phi:\h \times \h \to \a^{\ast}$, is the map: $\Phi(u,v)(x)=(-1)^{|x|(|u|+|v|)}B_{\h}(\rho(x)(u),v)$, then $|\Phi|=\delta$ and $\Phi_{\delta}=P_{\delta}(\Phi)$.
\smallskip

The tuple $(\h,[\,\cdot\,,\,\cdot\,]_{\h},B_{\h},\rho,\lambda,\omega_{\delta})$ is called a \emph{$\delta$-context of generalized double extension of $(\h,[\,\cdot\,,\,\cdot\,]_{\h},B_{\h})$ by $(\a,[\,\cdot\,,\,\cdot\,]_{\a})$}.

\begin{lema}\label{lema contexto}{\sl
Let $(\h,[\,\cdot\,,\,\cdot\,]_{\h},B_{\h},\rho,\lambda,\omega_{\delta})$ be a $\delta$-context of generalized double extension of $(\h,[\,\cdot\,,\,\cdot\,]_{\h},B_{\h})$ by $(\a,[\,\cdot\,,\,\cdot\,]_{\a})$. Then the maps $\chi_{\delta}$ and $\Phi_{\delta}$, satisfy the following for all $x,y$ in $\a$ and $u,v$ in $\h$:
\begin{equation}\label{lema 1}
\begin{split}
& \Phi_{\delta}(\rho(x)(u),v)+(-1)^{|x||u|}\Phi_{\delta}(u,\rho(x)(v))\\
& -\ad^{\ast}_{\delta}(x)(\Phi_{\delta}(u,v))-\chi_{\delta}(x,[u,v]_{\h})=0,\,\,\,\text{ and }
\end{split}
\end{equation}
\begin{equation}\label{lema 2}
\begin{split}
& \chi_{\delta}([x,y]_{\a},u)-\chi_{\delta}(x,\rho(y)(u))+(-1)^{|x||y|}\chi_{\delta}(y,\rho(x)(u))\\
& -\ad_{\delta}^{\ast}(x)(\chi_{\delta}(y,u))+(-1)^{|x||y|}\ad_{\delta}^{\ast}(y)(\chi_{\delta}(x,u))\\
& +\Phi_{\delta}(\lambda(x,y),u)=0.
\end{split}
\end{equation}
In addition, $\Phi_{\delta}:\h \times \h \to \mathbf{P}_{\delta}(\a)^{\ast}$ satisfies:
\begin{equation}\label{Phi es cociclo}
\sum_{\text{cyclic}}(-1)^{|u||w|}\Phi_{\delta}(u,[v,w]_{\h})=0.
\end{equation}
}
\end{lema}
\begin{proof}
We shall only give the details of \eqref{lema 2}. The expression given in \eqref{lema 2} consists of six elements where each of them belongs to $\mathbf{P}_{\delta}(\a)^{\ast}$. We shall determine these six elements by evaluating them in an arbitrary element $P_{\delta}(z)$ of $P_{\delta}(\a)$, where $z$ belongs to $\a$.
\smallskip

The first three elements in \eqref{lema 2} are of the form $\chi_{\delta}$. We use the definition of $\chi_{\delta}$ in $\eqref{chi delta}$ to determine these first three elements in \eqref{lema 2}:
\begin{equation}\label{demostracion lema 1}
\begin{split}
& \chi_{\delta}([x,y]_\a,u)(P_{\delta}(z))\!=\!-(-1)^{|u||z|}B_{\h}(\lambda([x,y]_\a,z),u),\\
& -\chi_{\delta}(x,\rho(y)(u))(P_{\delta}(z))\!=\!(-1)^{|x|(|y|+|z|)+|u||z|}B_{\h}(\rho(y)(\lambda(z,x)),u),\\
& (-1)^{|x||y|}\chi_{\delta}(y,\rho(x)(u))(P_{\delta}(z))\!=\!(-1)^{|z||u|}B_{\h}(\rho(x)(\lambda(y,z)),u).
\end{split}
\end{equation}
The fourth and fifth elements in \eqref{lema 2}, involve the representation $\delta$-coadjoint $\ad_{\delta}^{\ast}:\a \to \gl(P_{\delta}(\a)^{\ast})$, and the bilinear map $\chi_{\delta}$. We use that $\chi_{\delta}=P_{\delta}(\chi)$, $|\chi|=\delta$, and the definition of $\ad_{\delta}^{\ast}$ (see \eqref{coadjunta}), to obtain:
\begin{equation}\label{demostracion lema 2}
\begin{split}
& -\!\ad_{\delta}^{\ast}(x)(\chi_{\delta}(y,u))(P_{\delta}(z))\!=\!(-1)^{|x|(|y|+|z|)+|u||z|}\!B_{\h}(\lambda(y,[z,x]_{\a}),u),\\
& (-1)^{|x||y|}\ad_{\delta}^{\ast}(y)(\chi_{\delta}(x,u))(P_{\delta}(z))\!=\!(-1)^{|u||z|}\!B_{\h}(\lambda(x,[y,z]_{\a}),u).
\end{split}
\end{equation}
In the sixth element we apply \eqref{Phi es cociclo} and we obtain:
\begin{equation}\label{demostracion lema 3}
\Phi_{\delta}(\lambda(x,y),\!u)(P_{\delta}(z))\!=\!(-1)^{|z|(|x|+|y|)+|u||z|}B_{\h}(\rho(z)(\lambda(x,y)),\!u).
\end{equation}
Observe that each of the expressions in \eqref{demostracion lema 1}, \eqref{demostracion lema 2} and \eqref{demostracion lema 3}, $(-1)^{|u||z|}$ is a common factor. We add each of the terms appearing at the left of $B_{\h}(\,\cdot\,,u)$, given in \eqref{demostracion lema 1}, \eqref{demostracion lema 2} and \eqref{demostracion lema 3}, without $(-1)^{|u||z|}$, and we get:
\begin{equation}\label{demostracion lema 4}
\begin{split}
& \!-\!\lambda([x,y]_{\a},z)\!+\!(-1)^{|x|(|y|+|z|)}\rho(y)(\lambda(z,x))\!+\!\!\rho(x)(\lambda(y,z))\\
& \!+\!(\!-1)^{|x|(|y|+|z|)}\lambda(y,[z,x]_{\a})\!+\!\lambda(x,[y,z]_{\a})\!+\!(\!-1)^{|z|(|x|+|y|)}\rho(z)(\lambda(x,y))
\end{split}
\end{equation}
Multiplying \eqref{demostracion lema 4} by $(-1)^{|x||z|}$, we obtain that \eqref{demostracion lema 4} is equal to:
\begin{equation}\label{demostracion lema 5}
\sum_{\text{cyclic}}(-1)^{|x||z|}\left(\rho(x)(\lambda(y,z))+\lambda(x,[y,z]_{\a})\right).
\end{equation}
which is zero, by \eqref{deh2}. Similar arguments prove that \eqref{lema 1} is equivalent to \eqref{deh1}. Finally, the expression \eqref{Phi es cociclo} is equivalent to the fact that $\rho(x)$ is an homogeneous derivation of $(\h,[\,\cdot\,,\,\cdot\,]_{\h})$, for all $x$ in $\a$.
\end{proof}

\begin{Theorem}\label{teorema doble extension super}{\sl
Let $(\h,[\,\cdot\,,\,\cdot\,]_{\h},B_{\h},\rho,\lambda,\omega_{\delta})$ be a $\delta$-context of generalized double extension of $(\h,[\,\cdot\,,\,\cdot\,]_{\h},B_{\h})$ by $(\a,[\,\cdot\,,\,\cdot\,]_{\a})$, where $|B_{\h}|=\delta$. Then, in the super-space $\g=\a \oplus \h \oplus \mathbf{P}_{\delta}(\a)^{\ast}$, there exists a quadratic Lie super algebra structure of degree $\delta$.}
\end{Theorem}
\begin{proof}
By \textbf{Lemma \ref{lema contexto}}, the bilinear map $\Phi_{\delta}:\h \times \h \to \mathbf{P}_{\delta}(\a)^{\ast}$, is a 2-cocycle of $(\h,[\,\cdot\,,\,\cdot\,]_{\h})$ with values on the trivial $\h$-module $\mathbf{P}_{\delta}(\a)^{\ast}$. Let us consider the bracket $[\,\cdot\,,\,\cdot\,]^{\prime}$ on the super-space $\h \oplus \mathbf{P}_{\delta}(\a)^{\ast}$, defined by: $[u+P_{\delta}(\alpha),v+P_{\delta}(\beta)]^{\prime}=[u,v]_{\h}+\Phi_{\delta}(u,v)$, where $u+P_{\delta}(\alpha)$ and $v+P_{\delta}(\beta)$ are elements in $\h \oplus \mathbf{P}_{\delta}(\a)^{\ast}$. 
\smallskip

For each $x$ in $\a$, consider the map $\Theta(x):\h \oplus \mathbf{P}_{\delta}(\a)^{\ast} \to \h \oplus \mathbf{P}_{\delta}(\a)^{\ast}$, defined by:
\begin{equation}\label{teorema DEX1}
\Theta(x)(u+P_{\delta}(\alpha))=\rho(x)(u)+\ad_{\delta}^{\ast}(x)(P_{\delta}(\alpha))+\chi_{\delta}(x,u).
\end{equation}
Using that $\rho(x)$ is a derivation of $(\h,[\,\cdot\,,\,\cdot\,]_{\h})$, $\ad_{\delta}^{\ast}:\a \to \gl(\mathbf{P}_{\delta}(\a)^{\ast})$ is a representation, and the expression \eqref{lema 1}, it follows that $\Theta(x)$ is a derivation of $(\h \oplus \mathbf{P}_{\delta}(\a)^{\ast},[\,\cdot\,,\,\cdot\,]^{\prime})$, of degree $|x|$. Hence, we have the even linear map $\Theta:\a \to \Der(\h \oplus \mathbf{P}_{\delta}(\a)^{\ast})$, $x \mapsto \Theta(x)$. 
\smallskip

Let $\Lambda:\a \times \a \to \h \oplus \mathbf{P}_{\delta}(\a)^{\ast}$, be defined by: $\Lambda(x,y)=\lambda(x,y)+\omega_{\delta}(x,y)$, for all $x,y$ in $\a$. 
By \eqref{deh1} and \eqref{lema 2}, it follows:
\begin{equation}\label{teorema DEX2}
[\Theta(x),\Theta(y)]_{\gl(\h \oplus \mathbf{P}_{\delta}(\a)^{\ast})}-\Theta([x,y]_{\a})=\ad_{\h \oplus \mathbf{P}_{\delta}(\a)^{\ast}}(\Lambda(x,y)).
\end{equation}
Using \eqref{deh2} and \eqref{deh3} we get:
\begin{equation}\label{teorema DEX3}
\sum_{\text{cyclic}}(-1)^{|x||z|}\left(\Theta(x)(\Lambda(y,z))+\Lambda(x,[y,z]_{\a})\right)=0,
\end{equation}
for all $x,y,z$ in $\a$. By \eqref{teorema DEX2} and \eqref{teorema DEX3}, we consider the generalized semi-direct product $\g=\a \oplus \left(\h \oplus \mathbf{P}_{\delta}(\a)^{\ast}\right)$, of $\h \oplus \mathbf{P}_{\delta}(\a)^{\ast}$ by $\a$, by means of $(\Theta,\Lambda)$. Then the bracket $[\,\cdot\,,\,\cdot\,]$ of $\g$ takes the form:
\begin{equation}\label{corchete en doble extension}
\begin{split}
x,y \in \a,\,\quad & \,[x,y]=[x,y]_{\a}+\lambda(x,y)+\omega_{\delta}(x,y),\\
x \in \a,\,u \in \h,\quad & \,[x,u]=\rho(x)(u)+\chi_{\delta}(x,u),\\
u,v \in \h, \quad & \,[u,v]=[u,v]_{\h}+\Phi_{\delta}(u,v),\\
x \in \a,\,\alpha \in \a^{\ast},\quad &\,[x,P_{\delta}(\alpha)]=\ad_{\delta}^{\ast}(x)(P_{\delta}(\alpha))
\end{split}
\end{equation}
Let $x+u+P_{\delta}(\alpha),y+v+P_{\delta}(\beta)$ be homogeneous elements of $\g$. We define the non-degenerate, super symmetric bilinear form $B$ on $\g$ by:
\begin{equation}\label{metrica invariante en doble extension}
B(x+u+P_{\delta}(\alpha),y+v+P_{\delta}(\beta))=(-1)^{|x||\beta|}\beta(x)+B_{\h}(u,v)+\alpha(y).
\end{equation}
Then $B$ is invariant with $|B|=\delta$, and $(\g,[\,\cdot\,,\,\cdot\,],B)$ is a quadratic Lie super algebra of degree $\delta$.
\end{proof}
%%%%%%%%%%%%%%%%%%%%%%%%%%%%%%%%%%%%%%%%%%%%%%%%%%%%%%%%%%%%%
\section{Indecomposable Homogeneous quadratic Lie super algebras}
%%%%%%%%%%%%%%%%%%%%%%%%%%%%%%%%%%%%%%%%%%%%%%%%%%%%%%%%%%%

A quadratic Lie super algebra $\g$ is \textbf{decomposable} if there is a non-zero ideal $I \neq \g$, such that $\g=I \oplus I^{\perp}$. Otherwise, $\g$ is \textbf{indecomposable}.
\smallskip

Let $(\mathfrak{g},[\,\cdot\,,\,\cdot\,],B)$ be an indecomposable quadratic Lie algebra 
of degree $\delta$, no simple and $\dim\g>1$, over an algebraically closed field $\F$ of zero characteristic. Let $\mathcal{I}$ be a minimal ideal of 
$\mathfrak{g}$, then $\mathcal{I}\cap \mathcal{I}^{\perp} \neq \{0\}$ and consequently: $\mathcal{I} \subset \mathcal{I}^{\perp} \subset C_{\mathfrak{g}}(\mathcal{I})$; thus, 
$\mathcal{I}$ is Abelian. 
\smallskip

Let $\mathfrak{h}$ be a super subspace of $\mathfrak{g}$ such that: $\mathcal{I}^{\perp}=\mathfrak{h} \oplus \mathcal{I}$. Let $\mathfrak{a}$ be a super subspace satisfying: $\mathfrak{h}^{\perp}=\mathfrak{a} \oplus \mathcal{I}$. Then $\mathfrak{h}$ is a non-degenerate subspace 
of $\mathfrak{g}$ such that: $\mathfrak{g}=\mathfrak{h} \oplus \mathfrak{h}^{\perp}=\a \oplus \mathcal{I}^{\perp}$ and $\mathcal{I}^{\perp} \cap \mathfrak{h}^{\perp}=\mathcal{I}$.

\begin{claim}\label{claimJunio10}{\sl
There are bijective linear maps $\xi_{\delta}:\mathcal{I} \to \mathbf{P}_{\delta}(\a)^{\ast}$ and $\xi:\mathcal{I} \to \a^{\ast}$, such that:
$$
\textbf{(i)}\,\,|\xi_{\delta}|=0,\quad \textbf{(ii)}\,\,|\xi|=\delta,\quad \textbf{(iii)}\,\,P_{\delta}(\xi)=\xi_{\delta}.
$$  
}
\end{claim}
\begin{proof}
Let us consider the 
linear map $\xi_{\delta}:\mathcal{I} \rightarrow  \mathbf{P}_{\delta}(\mathfrak{a})^{\ast}$, defined by:
\begin{equation}\label{a9}
\xi_{\delta}(\alpha)\left(P_{\delta}(x)\right)=B(\alpha,x),\,\,\text{ for all }\alpha \in \mathcal{I},\,\text{ and }x \in \a.
\end{equation}
It is a straightforward to verify that $|\xi_{\delta}|=0$. Due to $B$ is non-degenerate, $\xi_{\delta}$ is injective, then: $\dim \mathcal{I} \leq \dim \mathbf{P}_{\delta}(\a)=\dim \a$. We know that $\dim(\mathcal{I}^{\perp}) \geq \dim(\mathfrak{g})-\dim(\mathcal{I})$, 
then $\dim(\mathcal{I}) \geq \dim(\mathfrak{a})$, and $\dim \mathcal{I}=\dim\a$.  Therefore, $\mathbf{P}_{\delta}(\mathfrak{a})^{\ast}$ is isomorphic to $\mathcal{I}$, via the map $\xi_{\delta}$. By {\bf Prop. \ref{proposicion 1s}} (see Appendix), we assume that $\mathfrak{a}$ is isotropic. Thus, we have the following decomposition of 
$\mathfrak{g}$: 
$$
\mathfrak{g}=\mathfrak{a} \oplus \mathfrak{h} \oplus \mathcal{I},\,\text{ where }\mathfrak{h}^{\perp}=\mathfrak{a} \oplus \mathcal{I},\,\text{ and }\mathbf{P}_{\delta}(\a)^{\ast} \text{ is isomorphic to } \mathcal{I}.
$$
If $\xi:\mathcal{I} \rightarrow \mathfrak{a}^{*}$ is the linear map defined by 
\begin{equation}\label{xi sin delta}
\xi(\alpha)(x)=B(\alpha,x),\,\,\text{ for all }\alpha \in \mathcal{I}, \text{ and }x \in \a,
\end{equation}
then $|\xi|=\delta$. In addition, by \textbf{Prop. \ref{propiedades de P}.(ii)}, $P_{\delta}(\xi)=\xi_{\delta}$.
\end{proof}

Let $B_{\mathfrak{h}}$ be the restriction of $B$ to $\h \times \h$. Using that $\g=\h \oplus \h^{\perp}$, and the even linear map $\xi_{\delta}$, the invariant metric $B$ takes the form:
\begin{equation}\label{descripcion de MI}
\begin{split}
& B(x+u+\alpha,y+v+\beta)=\\
&=(-1)^{|x||\beta|}\xi_{\delta}(\beta)(P_{\delta}(x))+\xi_{\delta}(\alpha)(y)+B_{\h}(u,v),
\end{split}
\end{equation}
where $x+u+\alpha$ and $y+v+\beta$ are homogeneous elements of $\g$.
\smallskip

Now we shall give a description of the bracket $[\,\cdot\,,\,\cdot\,]$, according to the decomposition $\g=\a \oplus \h \oplus \mathcal{I}$. 
\smallskip

Since $[\a,\mathcal{I}] \subset \mathcal{I}$, then for each $x$ in $\a$, there is a linear map $\sigma(x):\mathcal{I} \to \mathcal{I}$, such that: 
\begin{equation}\label{def representacion sigma}
\sigma(x)(\alpha)=[x,\alpha],\,\,\,\text{ for all }\,\alpha \in \mathcal{I}.
\end{equation}
Due to $[\,\cdot\,,\,\cdot\,]$ is even, the map $\sigma(x)$ is even. Thus \eqref{def representacion sigma} yields the even linear map $\sigma:\a \to \gl(\mathcal{I})$, $x \mapsto \sigma(x)$.
\smallskip

Let $x,y$ be in $\mathfrak{a}$. There are $[x,y]_{\mathfrak{a}}$ in $\mathfrak{a}$,
$\lambda(x,y)$ in $\mathfrak{h}$ and $\mu(x,y)$ in $\mathcal{I}$, such that: 
\begin{equation}\label{a no es subalgebra}
[x,y]=[x,y]_{\mathfrak{a}}+\lambda(x,y)+\mu(x,y).
\end{equation}
The properties of the Lie super bracket $[\,\cdot\,,\,\cdot\,]$, imply that the following maps are even, super skew-symmetric and bilinear: 
$$
\aligned
& \begin{array}{rccl}
[\,\cdot\,,\,\cdot\,]_{\mathfrak{a}}:&\mathfrak{a} \times \mathfrak{a}&
\longrightarrow &\,\,\mathfrak{a}\\
& (x,y) & \mapsto & [x,y]_{\a}
\end{array}\qquad 
 \begin{array}{rccl}
\lambda:&\mathfrak{a} \times \mathfrak{a}& \longrightarrow &\,\,\mathfrak{h}\\
& (x,y) & \mapsto & \lambda(x,y),
\end{array}\\
& \begin{array}{rccl}
\mu:&\mathfrak{a} \times \mathfrak{a}&
\longrightarrow & \,\,\mathcal{I}\\
& (x,y) & \mapsto & \mu(x,y)
\end{array}
\endaligned
$$

Let $x$ be in $\mathfrak{a}$ and $u$ be in $\mathfrak{h}$. Since $[\g,\h] \subset \mathcal{I}^{\perp}=\h \oplus \mathcal{I}^{\perp}$, there are $\rho(x,u)$ in $\mathfrak{h}$ and $\tau(x,u)$ in $\mathcal{I}$, such that:
\begin{equation}\label{1t}
[x,u]=\rho(x,u)+\tau(x,u).
\end{equation}
As the Lie super bracket $[\,\cdot\,,\,\cdot\,]_{\h}$ is even, then \eqref{1t} yields the even linear maps: $\rho(x):\h \to \h$, $u \mapsto \rho(x,u)$, and $\tau(x): \a \to \h$, $u \mapsto \tau(x,u)$, for all $x$ in $\a$ and $u$ in $\h$. Then, 
$$
\begin{array}{rccl}
\rho:&\a&\longrightarrow & \End(\h)\\
& x & \mapsto & \rho(x)
\end{array}
\quad \text{\and}\quad
\begin{array}{rccl}
\tau:&\a&\longrightarrow & \Hom(\h;\mathcal{I})\\
& x & \mapsto & \tau(x)
\end{array}
$$
are even linear maps. 
\smallskip

Let $u,v$ be in $\mathfrak{h}$. Since $\mathcal{I}^{\perp}=\h \oplus \mathcal{I}$, then there are $[u,v]_{\mathfrak{h}}$ in $\mathfrak{h}$ and $\gamma(u,v)$ in
$\mathcal{I}$, such that:
\begin{equation}\label{4}
[u,v]=[u,v]_{\mathfrak{h}}+\gamma(u,v)
\end{equation}
Let $[\,\cdot\,,\,\cdot\,]_{\h}:\h \times \h$ be the map $(u,v) \mapsto [u,v]_{\h}$ and let $\gamma:\h \times \h \to \mathcal{I}$ be the map $(u,v) \mapsto \gamma(u,v)$. Then, $[\,\cdot\,,\,\cdot\,]_{\h}$ and $\gamma$ are even, super skew-symmetric and bilinear.
\smallskip

In summary, the bracket $[\,\cdot\,,\,\cdot\,]$ of $\g$ can be described by the maps $\sigma,\lambda,\mu,\rho,\tau,[\,\cdot\,,\,\cdot\,]_{\h},\gamma$, as follows:
\begin{equation}\label{viernes7Junio}
\begin{split}
x,y \in \a,\,\quad & [x,y]=[x,y]_{\a}+\lambda(x,y)+\mu(x,y),\\
x \in \a,\,\,u \in \h,\,\quad & [x,u]=\rho(x)(u)+\tau(x)(u),\\
u,v \in \h,\,\quad & [u,v]=[u,v]_{\h}+\gamma(u,v),\\
x \in \a,\,\, \alpha \in \mathcal{I},\,\quad &[x,\alpha]=\sigma(x)(\alpha).
\end{split}
\end{equation}

\begin{claim}\label{a es superalgebra}{\sl
The pair $(\a,[\,\cdot\,,\,\cdot\,]_{\a})$ is a Lie super algebra satisfying:
\begin{align}
\label{17} & \sum_{\text{cyclic}}(-1)^{|x||z|}\left(\lambda(x,[y,z]_{\mathfrak{a}})+\rho(x)(\lambda(y,z))\right)=0,\,\,\text{ and }\\
\label{18} & \sum_{\text{cyclic}}(-1)^{|x||z|}\left(\mu(x,[y,z]_{\mathfrak{a}})+\tau(x)\left(\lambda(y,z)\right)+\sigma(x)(\mu(y,z))\right)=0.
\end{align}
}
\end{claim}
\begin{proof}
Let $x,y,z$ in $\mathfrak{a}$; by \eqref{viernes7Junio} we obtain:
\begin{equation}\label{16}
\begin{split}
[x,[y,z]]&=\underbrace{[x,[y,z]_{\mathfrak{a}}]_{\mathfrak{a}}}_{\in \a}+\underbrace{\lambda(x,[y,z]_{\mathfrak{a}})+\rho(x)\left(\lambda(y,z)\right)}_{\in \h}\\
\,&+\underbrace{\mu\left(x,[y,z]_{\mathfrak{a}}\right)+\tau(x)\left(\lambda(y,z)\right)+\sigma(x)\left(\mu(y,z)\right)}_{\in \mathcal{I}}.
\end{split}
\end{equation}
Uisng \eqref{16} and the Jacobi super identity of $[\,\cdot\,,\,\cdot\,]$ on $x,y,z$, we get that $(\mathfrak{a},[\,\cdot\,,\,\cdot\,]_{\mathfrak{a}})$ is a Lie super algebra satisfying \eqref{17} and \eqref{18}. 
\end{proof}

\begin{claim}\label{sigma es representacion}{\sl
The map $\sigma:\a \to \mathfrak{gl}(\mathcal{I})$ is a Lie super algebra representation of $(\a,[\,\cdot\,,\,\cdot\,]_{\a})$ on $\mathcal{I}$, isomorphic to the $\delta$-coadjoint representation $\ad_{\delta}^{\ast}:\a \to \gl(\textbf{P}_{\delta}(\a)^{\ast})$, via the map $\xi_{\delta}$, that is:
\begin{equation}\label{28}
\xi_{\delta} \circ \sigma(x)=\ad_{\boldsymbol{\delta}}^{\ast}(x)\circ \xi_{\delta},\quad \text{ for all } x \in \mathfrak{a}.
\end{equation}}
\end{claim}
\begin{proof}
Let $x,y$ be in $\mathfrak{a}$ and $\alpha$ be in $\mathcal{I}$. Using that $B$ is invariant and that $[\,\cdot\,,\,\cdot\,]$ is super skew-symmetric, we get:
\begin{equation}\label{coadjunta-1}
\xi_{\delta}\left(\sigma(x)(\alpha)\right)(P_{\delta}(y))=-(-1)^{|x||\alpha|}B(\alpha,[x,y]_{\mathfrak{a}}).
\end{equation}
Since $\xi_{\delta}(\alpha)=P_{\delta}\left(\xi(\alpha)\right)$ and $|\xi|=\delta$ (see \textbf{Claim \ref{claimJunio10}.(iii)}), we have:
\begin{equation}\label{coadjunta-2}
\begin{split}
& \ad_{\boldsymbol{\delta}}^{\ast}(x)\left(\xi_{\delta}(\alpha)\right)\left(P_{\delta}(y)\right)=\ad_{\delta}^{\ast}(x)\left(P_{\delta}(\xi(\alpha))\right)\left(P_{\delta}(y)\right)\\
& =-(-1)^{(|\xi(\alpha)|+\delta)|x|}\xi(\alpha)\left([x,y]_{\mathfrak{a}}\right)=-(-1)^{(\delta+|\alpha|+\delta)|x|}\xi(\alpha)\left([x,y]_{\mathfrak{a}}\right)\\
&=-(-1)^{|x||\alpha|}B(\alpha,[x,y]_{\a}).
\end{split}
\end{equation}
From \eqref{coadjunta-1} and \eqref{coadjunta-2}, we deduce \eqref{28}. Due to $\ad_{\delta}^{\ast}:\a \to \gl(\mathbf{P}_{\delta}(\a)^{\ast})$ is a representation of $(\a,[\,\cdot\,,\,\cdot\,]_{\a})$ (see \textbf{Prop. \ref{prop delta coadjunta}}) and $\xi_{\delta}$ is bijective, by \eqref{28} we deduce that $\sigma:\a \to \gl(\mathcal{I})$ is a representation of $(\a,[\,\cdot\,,\,\cdot\,]_{\a})$.
\end{proof}

\begin{claim}\label{condicion ciclica para mu}{\sl
Let $\mu_{\delta}=\xi_{\delta}\circ \mu:\a \times \a \to \mathbf{P}_{\delta}(\a)^{\ast}$. Then, $\mu_{\delta}(x,y)(P_{\delta}(z))=(-1)^{|x|(|y|+|z|)}\mu_{\delta}(y,z)(P_{\delta}(x))$, for all $x,y,z$ in $\a$.}
\end{claim}
\begin{proof}
Using \eqref{descripcion de MI} and \eqref{viernes7Junio}, we obtain:
$$
\aligned
& \mu_{\delta}(x,y)(P_{\delta}(z))=B(\mu(x,y),z)=B([x,y],z)=B(x,[y,z])\\
&=(-1)^{|x|(|y|+|z|)}B([y,z],x)=(-1)^{|x|(|y|+|z|)}B(\mu(y,z),x)\\
&=(-1)^{|x|(|y|+|z|)}\mu_{\delta}(y,z)(x),
\endaligned
$$
which proves our claim.
\end{proof}

\begin{claim}\label{tau-lambda}{\sl
Let $\chi_{\delta}:\a \times \h \to \mathbf{P}_{\delta}(\a)^{\ast}$, be defined by $\chi_{\delta}(x,u)=\xi_{\delta}(\tau(x)(u))$. Then $\chi_{\delta}(x,u)(P_{\delta}(y))=-(-1)^{|y||u|}B_{\h}(\lambda(x,y),u)$, for all $x,y$ in $\a$ and $u$ in $\h$.}
\end{claim} 
\begin{proof}
Using \eqref{descripcion de MI} and \eqref{viernes7Junio}, we obtain:
\begin{equation}\label{chi delta para teorema}
\begin{split}
& \chi_{\delta}(x,u)(P_{\delta}(y))=B(\tau(x)(u),y)=B([x,u],y)\\
&=-(-1)^{|x||u|}B(u,[x,y])=-(-1)^{|x||u|}B_{\h}(u,\lambda(x,y))\\
&=-(-1)^{|y||u|}B_{\h}(\lambda(x,y),u),
\end{split}
\end{equation}
which proves our result.
\end{proof}

\begin{claim}\label{h es cuadratica}{\sl
The triple $(\h,[\,\cdot\,,\,\cdot\,]_{\h},B_{\h})$ is a quadratic Lie super algebra of degree $\delta$.
}
\end{claim}
\begin{proof}
Let $u,v,w$ be in $\h$. By \eqref{viernes7Junio} and $[\h,\mathcal{I}]=\{0\}$, we obtain: $[u,[v,w]]=[u,[v,w]_{\h}]_{\h}\,\operatorname{mod}\mathcal{I}$. The Jacoby super identity of $[\,\cdot\,,\,\cdot\,]$ on $u,v,w$, implies that $(\h,[\,\cdot\,,\,\cdot\,]_{\h})$ is a Lie super algebra. 
\smallskip

Now we shall prove that $B_{\h}$ is invariant under $[\,\cdot\,,\,\cdot\,]_{\h}$. Since $\mathcal{I}\subset\h^{\perp}$, from \eqref{descripcion de MI} and \eqref{viernes7Junio} it follows: $B_{\h}(u,[v,w]_{\h})=B(u,[v,w])=B([u,v],w)=B_{\h}([u,v]_{\h},w)$. Then $B_{\h}$ is invariant under $[\,\cdot\,,\,\cdot\,]_{\h}$, and $(\h,[\,\cdot\,,\,\cdot\,]_{\h},B_{\h})$ is a quadratic Lie super algebra of degree $\delta$.
\end{proof}

\begin{claim}\label{rho es derivacion}{\sl
For each $x$ in $\a$, the map $\rho(x):\h \to \h$ is a $B_{\h}$-skew symmetric derivation of $(\h,[\,\cdot\,,\,\cdot\,]_{\h})$. 
} 
\end{claim}
\begin{proof}
Let $u,v$ be in $\h$. By \eqref{viernes7Junio} we have:
\begin{equation}\label{8}
[x,[u,v]_{\mathfrak{h}}]=\rho(x)([u,v]_{\mathfrak{h}})\,\operatorname{mod}\mathcal{I}.
\end{equation}

Since $[u,v]=[u,v]_{\h}\,\operatorname{mod}\mathcal{I}$, by \eqref{viernes7Junio} we obtain: 
\begin{equation}\label{viernes14Junio}
[x,[u,v]_{\h}]=[x,[u,v]]\,\operatorname{mod}\mathcal{I}
\end{equation}
Due to $[\mathcal{I},\mathcal{I}^{\perp}]=\{0\}$, and using the Jacobi super identity on $[x,[u,v]]$ in \eqref{viernes14Junio}, by \eqref{viernes7Junio} we get:
\begin{equation}\label{9}
[x,[u,v]_{\mathfrak{h}}]=[\rho(x)(u),v]_{\mathfrak{h}}+(-1)^{|x||u|}[u,\rho(x)(v)]_{\mathfrak{h}}\,\,\,\operatorname{mod} \mathcal{I}.
\end{equation}
Comparing the components in $\h$ of \eqref{8} and \eqref{9}, it follows:
\begin{equation}\label{10}
\rho(x)\left([u,v]_{\mathfrak{h}}\right)=[\rho(x)(u),v]_{\mathfrak{h}}+(-1)^{|x||u|}[u,\rho(x)(v)]_{\h}.
\end{equation}
Then $\rho(x)$ is a derivation of $(\h,[\,\cdot\,,\,\cdot\,]_{\h})$. Further, due to $B$ is invariant and $B(\h,\mathcal{I})=\{0\}$, by \eqref{descripcion de MI} and \eqref{viernes7Junio} we get: 
$$
B_{\mathfrak{h}}(\rho(x)(u),v)=-(-1)^{|x||u|}B_{\mathfrak{h}}(u,\rho(x)(v)),\,\text{ for all }u,v \in \h.
$$ 
Then $\rho(x)$ is a skew-symmetric map with respect to $B_{\h}$. 
\end{proof}
Let $x$ be in $\a$ and $u,v$ be in $\h$. Using \eqref{descripcion de MI} and \eqref{viernes7Junio}, we get:
$$
\aligned
& \xi_{\delta}(\gamma(u,v))(P_{\delta}(x))=B(\gamma(u,v),x)=B([u,v],x)\\
&=(-1)^{|x|(|u|+|v|)}B(x,[u,v])=(-1)^{|x|(|u|+|v|)}B_{\h}(\rho(x)(u),v)
\endaligned
$$
Let $\Phi_{\delta}=\xi_{\delta} \circ
\gamma:\mathfrak{h} \times \mathfrak{h} \rightarrow
\mathbf{P}_{\delta}(\mathfrak{a})^{\ast}$. Then,
\begin{equation}\label{Phi cociclo para teorema}
\Phi_{\delta}(u,v)(P_{\delta}(x))=(-1)^{|x|(|u|+|v|)}\,B_{\mathfrak{h}}(\rho(x)(u),v)
\end{equation}

\begin{claim}\label{NN2}{\sl
For all $x,y$ in $\a$ and $u$ in $\h$, the following holds:
\begin{equation}\label{23d}
\,[\rho(x),\rho(y)]_{\mathfrak{gl}(\h)}-\rho([x,y]_{\a})=\ad_{\h}(\lambda(x,y)).
\end{equation}
}
\end{claim}
\begin{proof}
Let $x,y$ be in $\a$ and $u$ be in $\h$. From \eqref{viernes7Junio}, we obtain:
\begin{equation}\label{NN1}
[[x,y]_{\a},u]=\rho([x,y]_{\a})(u)\,\operatorname{mod}\mathcal{I}.
\end{equation}
On the other hand, using \eqref{a no es subalgebra} and that $[\h,\mathcal{I}^{\perp}]=\{0\}$, we get:
\begin{equation}\label{19-A}
[[x,y]_{\a},u]=[[x,y],u]-[\lambda(x,y),u]
\end{equation}
In \eqref{19-A} above, we use the Jacobi super identity on $[[x,y],u]$ and that $[\lambda(x,y),u]=[\lambda(x,y),u]_{\h}\,\operatorname{mod}\mathcal{I}$ (see \eqref{viernes7Junio}), to obtain:
\begin{equation}\label{19-B}
[[x,y]_{\a},u]\!=\![x,[y,u]]\!-\!(-1)^{|x||y|}[y,[x,u]]\!-\![\lambda(x,y),u]_{\h}\,\operatorname{mod}\mathcal{I}
\end{equation}
In \eqref{19-B} we use that $[x,[y,u]]=\rho(x)\circ \rho(y)(u)\,\operatorname{mod}\mathcal{I}$ (see \eqref{viernes7Junio}), to get:
\begin{equation}\label{21}
\begin{split}
&\,[[x,y]_{\a},u]=\\
&\rho(x)\circ \rho(y)(u)-(-1)^{|x||y|}\rho(y)\circ \rho(x)(u)-[\lambda(x,y),u]_{\h}\,\,\operatorname{mod}\mathcal{I}
\end{split}
\end{equation}

Comparing the components in $\h$ of \eqref{NN1} and \eqref{21}, we get \eqref{23d}.
\end{proof}

In summary, by \textbf{Claim \ref{rho es derivacion}} and \textbf{Claim \ref{NN2}}, $\rho:\a \to \Der(\h)\cap \mathfrak{o}(B_{\h})$ is an even linear map satisfying:
$$
[\rho(x),\rho(y)]_{\gl(\h)}-\rho([x,y]_{\a})=\ad_{\a}(\lambda(x,y)).
$$ 
By \eqref{17}, the maps $\lambda$ and $\rho$ satisfy:
$$
\sum_{\text{cyclic}}(-1)^{|x||z|}\left(\rho(x)(\lambda(y,z))+\lambda(x,[y,z]_{\a})\right)=0.
$$
We also have an even bilinear map $\chi_{\delta}:\a \times \h \to \mathbf{P}_{\delta}(\a)^{\ast}$, satisfying: $\chi_{\delta}(x,u)(P_{\delta}(y))=(-1)^{|u||y|}B_{\h}(\lambda(x,y),u)$. On the other hand, applying $\xi_{\delta}$ to \eqref{18} and uisng that $\xi_{\delta}\circ \sigma(x)=\ad_{\delta}^{\ast}(x)\circ \xi_{\delta}$, we obtain:
$$
\sum_{\text{cyclic}}(-1)^{|x||z|}\left(\ad_{\delta}^{\ast}(x)(\omega_{\delta}(y,z))+\omega_{\delta}(x,[y,z]_{\a})+\chi_{\delta}(x,\lambda(y,z))\right)=0,
$$
where $\omega_{\delta}$ satisfies the super cyclic condition (see \textbf{Claim \ref{condicion ciclica para mu}}). By \eqref{Phi cociclo para teorema}, the bilinear map $\Phi:\h \times \h \to \mathbf{P}_{\delta}(\a)^{\ast}$, satisfies:
$$
\Phi_{\delta}(u,v)(P_{\delta}(x))=(-1)^{|x|(|u|+|v|)}B_{\h}(\rho(x)(u),v).
$$
Therefore, $(\h,[\,\cdot\,,\,\cdot\,]_{\h},B_{\h},\rho,\lambda,\omega_{\delta})$ is a context of generalized double extension of $(\h,[\,\cdot\,,\,\cdot\,]_{\h},B_{\h})$ by $(\a,[\,\cdot\,,\,\cdot\,]_{\a})$, in such a way that the homogeneous quadratic Lie super algebra in $\a \oplus \h \oplus \mathbf{P}_{\delta}(\a)^{\ast}$ of \textbf{Thm. \ref{teorema doble extension super}}, is isometric to $(\g,[\,\cdot\,,\,\cdot\,],B_{\h})$, via the map: $x+u+\alpha \mapsto x+u+\xi_{\delta}(\alpha)$.

\begin{cor}\label{corolario 1}{\sl
Let $(\mathfrak{g},[\,\cdot\,,\,\cdot\,],B)$ be an indecomposable, quadratic Lie super algebra of 
degree $\delta$, non-simple and
$\dim(\mathfrak{g})>1$. Then, $(\mathfrak{g},[\,\cdot\,,\,\cdot\,],B)$ is a generalized double 
extension of degree $\delta$, of a quadratic Lie super algebra $(\mathfrak{h},[\,\cdot\,,\,\cdot\,]_{\h},B_{\mathfrak{h}})$ 
of degree $\delta$, by a Lie super algebra $\left(\a,[\,\cdot\,,\,\cdot\,]_{\a}\right)$.}
\end{cor}

\subsection{Odd quadratic Lie super algebras}

We will use \textbf{Thm. \ref{teorema doble extension super}}, to recover the odd quadratic Lie super algebra studied in \cite{Alb}.
\smallskip

Let$\left(\mathfrak{h},[\,\cdot\,,\,\cdot\,]_{\mathfrak{h}},B_{\mathfrak{h}}\right)$ be an odd quadratic Lie super algebra and consider a context of generalized double extension $(\mathfrak{h},[\,\cdot\,,\,\cdot\,]_{\h},B_{\mathfrak{h}},\rho,\lambda,\omega_1)$, of $(\mathfrak{h},[\,\cdot\,,\,\cdot\,]_{\mathfrak{h}},B_{\mathfrak{h}})$ by $\left(\mathfrak{a},[\,\cdot\,,\,\cdot\,]_{a}\right)$, where $\dim(\mathfrak{a})=1$, $\mathfrak{a}=\F\,x$ and $|x|=1$. 
\smallskip

By \textbf{Thm. \ref{teorema doble extension super}}, there is an odd quadratic Lie super algebra on the super-space $\g=\a \oplus \h \oplus \mathbf{P}(\a)^{\ast}$. Since $\rho:\a \to \Der(\h) \cap \mathfrak{o}(B_{\h})$ is even then $D=\rho(x)$ is an odd skew-symmetric derivation of $\h$. Let $w=\lambda(x,x)$, which is an even element of $\h$.
\smallskip

From \eqref{deh1}, we obtain $2D^2=\ad_{\mathfrak{h}}(w)$, then $D^2=\frac{1}{2}\ad_{\mathfrak{h}}(w)$. From \eqref{deh2}, we have that $D(\lambda(x,x))=D(w)=0$. In addition, from \eqref{chi delta} we obtain:
$$
\chi_1(x,u)(P(x))=-(-1)^{|u|}B_{\mathfrak{h}}(u,w),\,\text{ for all }u \in \h.
$$
Then, $\chi_1(x,u)=-(-1)^{|u|}B_{\mathfrak{h}}(u,w)P(x)^{*}$, for all $u$ in $\h$. By \eqref{deh4} we get: $\Phi_1(u,v)(P(x))=(-1)^{|u|+|v|}B_{\mathfrak{h}}(D(u),v)$; hence $\Phi_1(u,v)=(-1)^{|u|+|v|}B_{\mathfrak{h}}(D(u),v)P(x)^{*}$. Let $\varphi\!:\!\mathfrak{h} \!\times \!\mathfrak{h}\!\to
\!\mathbf{P}(\mathbb{F})$ be defined by:
$$
\varphi(u,v)=B_{\mathfrak{h}}(D(u),v),\quad \text{ for all }u,v \in \h.
$$
If $|u|=|v|$, then $\Phi_1(u,v)(P(x))=\varphi(u,v)$. If $|u| \neq |v|$, then
$|D(u)|=|v|$ and, consequently: $\varphi(u,v)=\Phi_1(u,v)(P(x))=0$. Thus,
$\varphi(u,v)=\Phi_1(u,v)(P(x))$. We extend the derivation $D$ to $\h \oplus \mathbf{P}(\a)^{\ast}$, by defining $\widehat{D}:\h \oplus \mathbf{P}(\a)^{\ast} \to \h \oplus \mathbf{P}(\a)^{\ast}$, as follows:
$$
\aligned
& \widehat{D}(u)=D(u)-(-1)^{|u|}B_{\mathfrak{h}}(u,w)P(x)^{\ast},\quad \text{ for all }u \in \h,\,\text{ and }\\
& \widehat{D}(P(x)^{*})=0.
\endaligned
$$
Let $\omega_1(x,x)=\eta \, P(x)^{*}$, where $\eta$ is in $\F$. Then, by \textbf{Thm. \ref{teorema doble extension super}}, the bracket $[\,\cdot\,,\,\cdot\,]$ on $\g=\a \oplus \h \oplus \mathbf{P}(\a)^{*}$, is given by:
$$
\aligned
\,& [x,x]=w+\eta\, P(x)^{*},\\
\,& [x,u]=D(u)-(-1)^{|u|}B_{\mathfrak{h}}(u,w)P(x)^{*},\,\text{ for all } \, u \in \h,\\
\,& [u,v]=[u,v]_{\h}+B_{\mathfrak{h}}(D(u),v),\,\,\text{ for all } \, u,v \in \h,\\
\,& [x,P(x)^{*}]=0.
\endaligned
$$
The invariant metric $B:\g \times \g \to \F$ is given by:
$$
\aligned
& B(u,v)=B_{\mathfrak{h}}(u,v),\,\,\,\text{ for all } \,u,v \in \mathfrak{h},\\
& B(x,P(x)^{*})=1,\quad \text{ and }B(\mathfrak{h},x)=B(\mathfrak{h},P(x)^{*})=\{0\}.
\endaligned
$$

%%%%%%%%%%%%%%%%%%%%%%%%%%%%%%%%%%%%%%%%%%%%%%%%%%%%%%
\subsection{Heisenberg Lie super algebra extended by a derivation}
%%%%%%%%%%%%%%%%%%%%%%%%%%%%%%%%%%%%%%%%%%%%%%%%%%%%%%%%%

There are homogeneous quadratic Lie super algebras which can not be recover from the results given in \cite{Alb} and \cite{Ben}. One of them is the Heisenberg Lie super algebra extended by an even derivation. We will show that with the results in \S 3, we can recover this structure. 
\smallskip

Let $\left(\mathfrak{h},[\,\cdot\,,\,\cdot\,]_{\h},B_{\mathfrak{h}}\right)$ be an odd quadratic Lie super algebra, and consider a context of generalized double extension
$(\mathfrak{h},[\,\cdot\,,\,\cdot\,]_{\h},B_{\mathfrak{h}},\rho,\lambda,\omega_1)$ of $(\mathfrak{h},[\,\cdot\,,\,\cdot\,]_{\mathfrak{h}},B_{\mathfrak{h}})$ by $(\mathfrak{a},[\,\cdot\,,\,\cdot\,]_{\a})$. We assume that 
$\dim(\mathfrak{a})=1$, and $\mathfrak{a}=\F\,x$, with
$|x|=0$. By \textbf{Thm. \ref{teorema doble extension super}}, there is an odd quadratic Lie super algebra structure on $\g=\a \oplus \h \oplus \mathbf{P}(\a)^{\ast}$, that below we describe. 
\smallskip

Let $\rho:\a \to \Der(\h) \cap \mathfrak{o}(B_{\h})$ be the even linear map of \eqref{deh1} and let $D=\rho(x)$. As $|x|=0$, $D$ is an even skew-symmetric derivation of $\h$. Observe that \eqref{deh1} and \eqref{deh2} are trivially satisfied.
In addition, $\chi_1=\omega_1=0$. From \eqref{deh4}, we have: $\Phi_1(u,v)\left(P(x)\right)=B_{\mathfrak{h}}(D(u),v)$, for all $u,v$ in $\h$. Thus, the Lie super bracket $[\,\cdot\,,\,\cdot\,]$, on $\mathfrak{g}$, is given by:
\begin{equation*}
[\eta x\!+\!u\!+\!\zeta P(x)^{*},\eta^{\prime} x\!+\!v\!+\!\zeta^{\prime}
P(x)^{*}]=\![u,v]_{\mathfrak{h}}\!+\!D(u)\!-\!D(v)\!+\!B_{\mathfrak{h}}(D(u),v)P(x)^{*}
\end{equation*}
for all $u,v$ in $\h$ and $\eta,\eta^{\prime},\zeta,\zeta^{\prime}$ in $\F$. The invariant metric $B$ on $\mathfrak{g}$, is:
$$
\aligned
& B(u,v)=B_{\mathfrak{h}}(u,v),\,\,\text{ for all } \, u,v \in \mathfrak{h},\\
& B(\h,x)=B(\h,P(x)^{*})=\{0\},\quad \text{ and }B(x,P(x)^{*})=1.
\endaligned
$$
In the case that $\h$ is Abelian and the bilinear map $\omega:\mathfrak{h} \times \mathfrak{h} \rightarrow \mathbb{F}$, $(u,v) \mapsto B_{\mathfrak{h}}(D(u),v)$ is non-degenerate, then $\g$ is isometric to the the Heisenberg Lie super algebra $\h(D)=\F D \oplus \h \oplus \F \hslash$ extended by an even derivation $D:\h \to \h$, (see \cite{Mac}), via the map:
\begin{equation*}
\begin{array}{rccl}
\Psi:& \mathfrak{g} & \longrightarrow & \mathfrak{h}(D)\\
& \eta\, x+u+\zeta\, P(x)^{*} & \mapsto & \eta\, D+u+\zeta\, \hslash,
\end{array}
\end{equation*}

\section*{Appendix}

\begin{proposicion}\label{proposicion 1s}{\sl
Let $\g$ be a $m$-dimensional super-space over an algebraically closed field $\F$ and let
$B:\g \times \g \rightarrow \mathbb{F}$ be a non-degenerate, super symmetric, homogeneous and bilinear form. Let 
$\mathcal{I}=\mathcal{I}_0 \oplus \mathcal{I}_1$ be a $r$-dimensional, isotropic super sub-space of $\g$. Then, there exists a super sub-space $\a$ of $\g$, such that:
\smallskip

\textbf{(i)} $\a$ is isotropic.
\smallskip

\textbf{(ii)} $\dim(\a)=\dim(\mathcal{I})$.
\smallskip

\textbf{(iii)} $\a \cap \mathcal{I}=\{0\}$.
\smallskip

\textbf{(iv)} $\a \oplus \mathcal{I}$ is a non-degenerate super sub-space of $\g$ where $\mathcal{I} \simeq \mathbf{P}_{\delta}(\a)^{\ast}$.
}
\end{proposicion}
\begin{proof}
We assume that $B$ is odd. Due to $B$ is non-degenerate, then $\g_0$ and $\g_1$ are isotropic sub-spaces of $\g$, $\g_0 \simeq \g_1^{\ast}$ and $\dim(\g)=m=2n$, where $n=\dim \g_0=\dim \g_1$. 
\smallskip

Let $\mathcal{B}_{I}\!=\!\{x_1,\ldots,x_s,y_{1},\ldots,y_{t}\}$ be a basis of $\mathcal{I}$, where $\mathcal{I}_0\!=\!\operatorname{Span}_{\F}\{x_1,\ldots x_s\}$ and $\mathcal{I}_1=\operatorname{Span}_{\F}\{y_1,\ldots y_t\}$. Let $U=U_0 \oplus U_1$ be a super sub-space such that $\g=\mathcal{I} \oplus U$. 
\smallskip

Since $B$ is non-degenerate, there are linearly independent and homogeneous elements: $f_1+u_1,\ldots,f_n+u_n,g_1+v_1,\ldots,g_n+v_n$ in $\g$, where $f_i+u_i$ belongs to $\mathcal{I}_0\oplus U_0$, and $g_i+v_i$ belongs to $\mathcal{I}_1 \oplus U_1$, satisfying: 
\begin{equation}\label{vectores duales}
\begin{split}
& B(x_i,v_j)=B(x_i,g_j+v_j)=\delta_{ij},\quad 1 \leq i \leq s,\,\,\, 1 \leq j \leq n,\,\text{ and }\\
& B(y_i,u_j)=B(y_i,f_j+u_j)=\delta_{ij}, \quad 1 \leq i \leq t,\,\,\,1 \leq j \leq n.
\end{split}
\end{equation}
We define the odd vectors $w_1^{\prime},\ldots,w_s^{\prime}$, by:
\begin{equation}\label{definicion vectores impares}
w^\prime_k=v_k-B(v_k,u_1)y_1-\ldots-B(v_k,u_t)y_t,\quad 1 \leq k \leq s.
\end{equation}
Since $B$ is odd, $B(w^{\prime}_k,w^{\prime}_{\ell})=0$, for all $1 \leq k,\ell \leq s$. On the other hand, by \eqref{vectores duales} we get:
\begin{equation}\label{J isotropico}
B(w^{\prime}_k,u_{\ell})=0,\quad \text{ for all }1 \leq k \leq s,\,\text{ and }\,1 \leq \ell \leq t.
\end{equation}
Let $w_{\ell}=u_{\ell}$ for $1 \leq \ell \leq t$ and let $\a=\operatorname{Span}_{\F}\{w_1,\ldots,w_t, w^{\prime}_1,\ldots,w^{\prime}_s\}$, where $\a_0=\operatorname{Span}_{\F}\{w_1,\ldots,w_t\}$ and $\a_1=\operatorname{Span}_{\F}\{w^{\prime}_1,\ldots,w^{\prime}_t\}$. Since $\mathcal{I}$ is isotropic, by \eqref{vectores duales} and the definition of $w^{\prime}_k$ given in \eqref{definicion vectores impares}, we have:
\begin{equation}\label{I y J identificados}
\begin{split}
B(w_i,y_j)&=B(u_i,y_j)=\delta_{ij},\quad 1 \leq i,j \leq t,\\
B(w^{\prime}_i,x_j)&=B(v_i,x_j)=\delta_{ij},\quad 1 \leq i,j,\leq s.
\end{split}
\end{equation}
Hence \eqref{J isotropico}-\eqref{I y J identificados}, $\a$ is isotropic, $\a \cap \mathcal{I}=\{0\}$ and $\mathcal{I}$ is isomorphic to $\mathbf{P}_{1}(\a)^{\ast}$ via the map $x_i \mapsto B^{\flat}(x_i)\vert_{\a}$ and $y_i \mapsto B^{\flat}(y_i)\vert_{\a}$.
\end{proof}

\section*{Acknowledgements}

These results are contained in the PhD. thesis of the author. The author thanks the support 
provided by CONAHCYT Grant 33605 and post-doctoral fellowship
CONAHCYT 769309. The author has no conflicts to disclose.

\end{document}